\documentclass[12pt]{amsart}
\usepackage{amsmath, amssymb,epic,graphicx,mathrsfs,enumerate}
\usepackage[all]{xy}
\usepackage{color}
\usepackage{comment}

\usepackage{amsthm}
\usepackage{amssymb}
\usepackage{latexsym}
\usepackage{longtable}
\usepackage{epsfig}
\usepackage{amsmath}
\usepackage{hhline}

\DeclareMathOperator{\inn}{Inn} \DeclareMathOperator{\perm}{Sym}
 \DeclareMathOperator{\soc}{soc}
\DeclareMathOperator{\aut}{Aut}

 \DeclareMathOperator{\frat}{Frat}

\DeclareMathOperator{\maxd}{MaxDim}
\DeclareMathOperator{\sym}{Sym}

\DeclareMathOperator{\gl}{GL} 
 
\DeclareMathOperator{\GL}{GL} \DeclareMathOperator{\AGL}{AGL}

\DeclareMathOperator{\core}{Core}

\DeclareMathOperator{\alt}{Alt}

\DeclareMathOperator{\fit}{Fit}

\newcommand{\FF}{\mathbb F}

\renewcommand{\emptyset}{\varnothing}

\DeclareMathOperator{\menta}{MinInt}
\DeclareMathOperator{\manta}{MaxInt}
\DeclareMathOperator{\mindim}{MinDim}
\DeclareMathOperator{\maxdim}{MaxDim}

\newtheorem{thm}{Theorem}
\newtheorem{cor}[thm]{Corollary}
 \newtheorem{lemma}[thm]{Lemma}
\newtheorem{prop}[thm]{Proposition} 
 \newtheorem{defn}[thm]{Definition}
\newtheorem{question}[thm]{Question}

\numberwithin{equation}{section}

\renewcommand{\footnote}{\endnote}
\newcommand{\ignore}[1]{}\makeglossary

\begin{document}
	\bibliographystyle{amsplain}
	\title[Maximal intersections]{Maximal intersections in  finite groups}

	\author{Andrea Lucchini}
	\address{Universit\`a degli Studi di Padova\\  Dipartimento di Matematica \lq\lq Tullio Levi-Civita\rq\rq\\ Via Trieste 63, 35121 Padova, Italy\\email: lucchini@math.unipd.it}
	\thanks{We would like to thank Silvio Dolfi, Daniele Garzoni and Attila Mar\'{o}ti for fruitful discussions and valuable and helpful comments.}

	\begin{abstract} For a finite group $G$, we investigate the behaviour of four invariants, $\text{MaxDim}(G),$ $\text{MinDim}(G),$ $\text{MaxInt}(G)$ and $\text{MinInt}(G),$ measuring in some way the width and the height of the lattice  $\mathcal M(G)$ consisting of the intersections of the maximal subgroups of $G.$	\end{abstract}
	\maketitle

We will say that a subgroup $H$ of a finite group $G$ is a maximal intersection in $G$ if there exists a family $M_1,\dots,M_t$ of maximal subgroups of $G$ with $H=M_1\cap \dots \cap M_t.$ We will denote by $\mathcal M(G)$  the subposet of the subgroup lattice of $G$ consisting of $G$ and all the maximal intersections in $G$.

Let $\mathcal X$ be a set of maximal subgroups of the finite group $G.$ We say that $\mathcal X$ is irredundant if the intersection of the subgroups in $\mathcal X$ is not equal to the intersection of any proper subset of $\mathcal X.$ The maximal dimension $\maxdim(G)$ of $G$ is defined as the maximal size of an irredundant set of maximal subgroups of $G.$ This definition arises from the study of the maximum size $m(G)$ of an irredundant generating set for $G$ (that is, a generating set that does not properly contain any other generating set). Indeed, it is easy to see that $m(G) \leq  \maxdim(G)$. However in \cite{delu} it was proved that the difference $\maxdim(G)-m(G)$ can be arbitrarily large. Independently on this motivation, $\maxdim(G)$ can be view as a measure of the width of the poset $\mathcal M(G).$ The study of the maximal dimension of an arbitrary finite group is quite complicated and it is difficult to find good strategies to investigate this invariant. For example one could ask the following question.
\begin{question}\label{quone} Let $N$ be a normal subgroup of a finite group $G.$ Suppose that $\delta=\maxdim{G/N},$ $d=\maxdim{G}$ and  $\{M_1/N,\dots,$ $M_\delta/N\}$ is an irredundant set of maximal subgroups of $G/N$. Do there exist $d-\delta$ maximal subgroups $M_{\delta+1},\dots,M_d$ of $G$ such that $M_1,\dots,M_\delta,$ $M_{\delta+1},\dots,M_d$ is an irredundant set of maximal subgroups of $G$?
\end{question} We will prove that the answer is negative and this makes difficult to estimate $\maxdim(G)$ arguing by induction. Another natural question to which we will give an unexpected negative answer is the following. 
\begin{question}\label{quodue}Does $G$ contain an irredundant family of maximal subgroups of size $\maxdim(G)$ whose intersection is the Frattini subgroup?
\end{question}

The dual concept of minimal dimension was introduced in \cite{GL}. We say that an irredundant set of maximal subgroups is maximal irredundant if it is not properly contained in any other irredundant set of maximal subgroups. Then the minimal dimension of $G,$ denoted $\mindim(G)$, is the minimal size of a maximal irredundant set.  
In \cite{bgl} it was proposed to study the finite groups $G$ with $\mindim(G)=\maxdim(G)$ (minmax groups). All nilpotent groups are minmax, but there are non-nilpotent minmax groups, such as Sym(3), Alt(4) and Sym(4).  By a well-known theorem of Iwasawa \cite{iwa}, all unrefinable chains in the subgroup lattice of a finite group $G$ have the same length if and only if $G$ is supersoluble. In our case, in place of arbitrary unrefinable chains, we restrict our attention to the unrefinable chains in the poset $\mathcal M(G).$  In the context of Iwasawa's result, it is worth noting that supersolubility does not imply the minmax property. However it seems reasonable to conjecture that every minmax group is soluble. The results in \cite{bgl} give evidence to this conjecture, proving that every non-abelian finite simple group is not minmax.

  One could even expect to have a relation between $\mindim(G),$ $\maxdim(G)$  and the minimal and maximal length of an unrefinable chain in $\mathcal M(G)$. We call these two new invariants $\menta(G)$ and $\manta(G)$. However the behavior of these invariants is more intricate than one can expected. It is not difficult to prove that the following holds.
  \begin{thm}\label{facile}
If  $G$ is a finite group, then 
\begin{enumerate}
\item $\mindim(G)\leq \menta(G);$ 
\item $\maxdim(G)\leq \manta(G).$
\end{enumerate}
  \end{thm}
  
   However the differences $\menta(G)-\mindim(G)$ and $\manta(G)-\maxdim(G)$ can be arbitrarily large. We will see in Section \ref{superss} that for any pair $(a,b)$ of positive integers with $2\leq a\leq b,$ it can be constructed a finite soluble group $G$ with $\mindim(G)=a+b,$ $\menta(G)=2a+b,$ $\maxdim(G)=a+b,$ $\manta(G)=a+2b.$ 
   
   \
   
   It is more difficult to compare $\maxdim(G)$ and $\menta(G).$ In the example mentioned above $\menta(G)-\maxdim(G)=a$ can be chosen to be arbitrarily large. However  if $p$ is a prime, then $\lim_{p\to \infty}\maxdim(\perm(p))-\menta(\perm(p))=\infty.$
   
 \begin{defn}
 	We say that a finite group $G$ is strongly minmax if $\maxdim(G)=\menta(G)$. 
 \end{defn}
 By Theorem \ref{facile}, if $G$ is strongly minmax then $\mindim(G)=$ $\maxdim(G)$ $=\menta(G)=\manta(G).$  This occurs for example when $G$ is nilpotent or, more in general, if there exists a finite nilpotent group $K$ with $\mathcal M(G)\cong \mathcal M(K).$ 
 The group with these property have been studied in \cite{join}, where it is proved in particular that they are supersoluble. Notice that $\sym(4)$ is a strongly minmax group
 which is not supersoluble.
 
 \begin{defn}Let $G$ be a finite group. We define $\alpha(G)$ as the smallest cardinality of a family of maximal subgroups of $G$ with the property that their intersection coincide with the Frattini subgroup of $G$.
 \end{defn}

Clearly $\mindim(G)\leq \alpha(G)\leq \manta(G).$ In particular, if $G$ strongly minmax, then
$\mindim(G) = \alpha(G) = \manta(G).$ This motivates the following definition.
 
 \begin{defn}
 	We say that a finite group $G$ is weakly minmax if $\menta(G)=\manta(G)=\alpha(G).$
 \end{defn}

   Our main theorem is the following.
   \begin{thm}
   	If $G$ if a finite weakly minmax group, then $G$ is soluble. Moreover  the derived length of $G/\frat(G)$ is at most 3.
   \end{thm}

The bound 3 on the derived length is best possible, since, for example, $\perm(4)$ is weakly minmax.	Notice that $\menta(\alt(5))=\manta(\alt(5))=3,$ hence the condition  $\menta(G)=\manta(G)$ does not imply that $G$ is soluble.

\section{Negative answers to questions \ref{quone} and \ref{quodue}}

Our first aim is to give a negative answer to question \ref{quone}.
Let $\FF$ be the field with $11$ elements and let $C=\langle c \rangle $ be the subgroup of order $5$ of the multiplicative group of $\FF.$
Let $V=\FF^5$
be a $5$-dimensional vector space over $\FF$ and let $\sigma=(1,2,3,4,5)\in \perm(5).$  The wreath group $H=C\wr \langle \sigma \rangle$ has an irreducible action on $V$ defined as follows:
if $v=(f_1,\dots,f_5)\in V$ and $h=(c_1,\dots,c_5)\sigma \in H$, then $v^h=(f_{1\sigma^{-1}}c_{1\sigma^{-1}},\dots,f_{5\sigma^{-1}}c_{5\sigma^{-1}}).$
We will concentrate our attention on the semidirect product
$G:=V\rtimes H$ (notice that $G=G_{11,5}$ in the notations of \cite[Section 3]{delu}). By \cite[Proposition 11]{delu} $\maxd(G)=5$, while $H\cong G/V$ is a 2-generated nilpotent group with  
$\maxd(H)$=2. Let $M_1, M_2$ be two different maximal subgroups of $G$ containing $V.$ We have $M_1=V\rtimes K_1$ and $M_2=V\rtimes K_2$ with $K_1$ and $K_2$  maximal subgroups of $H.$ Assume, by contradiction, that $\{M_1,M_2\}$ can be lifted to an irredundant  set $\{M_1,M_2,M_3,M_4,M_5\}$ of maximal subgroups of $G.$ Then  $M_3,M_4,M_5$ are complements of $V$ in $G$ and it is not restrictive to assume $M_3=H,$ $M_4=H^{v_1},$ $M_5=H^{v_2}$ with $v_1=(x_1,x_2,x_3,x_4,x_5), v_2=(y_1,y_2,y_3,y_4,y_5)\in V.$ We must have $|M_3\cap M_4|\geq 5^3,$ hence, by \cite[Lemma 10]{delu}, the subset $I$ of $\{1,\dots,5\}$ consisting of the indices $i$ with $x_i=0$  contains at least 3 elements and $M_3\cap M_4=\{(c_1,c_2,c_3,c_4,c_5)\in C_5^5\mid c_i=1 \text{ if }i\neq I\}.$ Notice that $M_1\cap M_2=V\rtimes F$ with $F=\frat H=
\{(c_1,c_2,c_3,c_4,c_5)\in C_5^5\mid c_1c_2c_3c_4c_5=1\}.$ In particular
$M_1\cap M_2\cap M_3\cap M_4=\{(c_1,c_2,c_3,c_4,c_5)\in C_5^5\mid c_1c_2c_3c_4c_5=1 \text { and }c_i=1 \text{ if }i\neq I\},$
but then $|M_3\cap M_4:M_1\cap M_2\cap M_3\cap M_4|=5$, a contradiction.

\

Now we give a negative answer to question \ref{quodue}. Let $G=\AGL(2,5),$  $N
=\soc(G)$ and $F/N=\frat(G/N).$ We have $N\cong C_5\times C_5$, $F/N\cong \frat(\gl(2,5))\cong C_4$ and $G/F\cong \perm(5).$ For a maximal subgroup $M$ of $G$ we have the following possibilities:
\begin{enumerate}
	\item $M$ is a complement of $N$ in $G$ (25 conjugates);
	\item $F\leq M$ and $M/F \cong \alt(5)$ (1 conjugate);
	\item $F\leq M$ and $M/F \cong \perm(4)$ (5 conjugates);
	\item $F\leq M$ and $M/F \cong C_5\rtimes C_4$ (6 conjugates);
	\item $F\leq M$ and $M/F \cong \perm(3)\times \perm(2)$ (10 conjugates).
\end{enumerate}
Let $H$ be a complement of $N$ in $G$ and let $\{K_1,\dots,K_4\}$ be an irredundant family of maximal subgroups of $G$ of type 3.
Then  $\{H\cap K_1,\dots,H\cap K_4\}$ is an irredundant family of maximal subgroups of $H$  and $H\cap K_1\cap \dots \cap K_4=H\cap F=\frat(H)\cong C_4.$
In particular $\maxd(G)\geq 5.$ Now assume that $\mathcal M=\{M_1,\dots,M_t\}$ is an irredundant family of maximal subgroups of $G$ with $M_1\cap \dots \cap M_t=1.$ At least one of these maximal subgroups of $G$ must be of type (1), otherwise $M_1\cap \dots \cap M_t\geq F.$ So it is not restrictive to assume $M_1=H.$ We distinguish two possibilities:

\noindent a) $M_i$ is not of type (1), whenever $i\geq 2.$ In this case, for $i\geq 2$ there exists a maximal subgroup $K_i$ of $H$ such that $M_i=NK_i$. But then $M_1\cap M_2 \cap \dots \cap M_t= K_2\cap \dots \cap K_t\geq \frat(H)\cong C_4,$ a contradiction.

\noindent b) $M_2$ is of type (1). We have $M_2=H^n$ for some $1\neq n\in N$ and $D=M_1\cap M_2=H\cap H^n=C_H(n)\cong C_5\rtimes C_4.$ If $X$ is a maximal subgroup of $G$ of type (3), then $X\cap D\cong C_4$, and consequently,  if $X_1$ and $X_2$ are two different maximal subgroups of type (3), then $X_1\cap X_2\cap D =1$. Hence either $t\leq 4$ or $\mathcal M$ contains at most one maximal subgroup of type (3). One of the 6 maximal subgroups of $G$ of type (4) contains $D,$ the other intersect $D$ in a subgroup of order $4$, moreover if  $Y_1$ and $Y_2$ are two different maximal subgroups of type (4) not containing $D$, then $Y_1\cap Y_2\cap D =1$. Hence either $t\leq 4$ or $\mathcal M$ contains at most one maximal subgroup of type (4). If $Z$ is a maximal subgroup of $G$ of type (5), then $Z\cap D\cong C_2$.
Hence if $\mathcal M$ contains a maximal subgroup of type (5) then $t\leq 4.$
Summarizing we have proved that either $t\leq 4 <\maxd(G)$ or $t=5$ and in that case we may assume $M_3$ of type (2), $M_4$ of type (3) and $M_5$ of type (4).
However this case cannot occur since it can be easily checked that if $X$ is of type (3) and $Y$ is of type (4) and does not contains $D$, then either $X\cap Y\cap D=1$ or $X\cap D=Y\cap D.$

\

Although question \ref{quodue} has a negative answer, a weaker result holds.

\begin{prop}\label{frate}
 If $\{M_1,\dots,M_t\}$ is a maximal irredundant family of maximal subgroups, then $\core_G(M_1\cap\dots\cap M_t)=\frat(G).$
\end{prop}

\begin{proof}
	It is not restrictive to assume $\frat(G)=1.$ Let $\{M_1,\dots,M_t\}$ be a maximal irredundant family of maximal subgroups of $G$ and let $X=
	M_1\cap\dots\cap M_t.$ Assume, by contradiction, that $X$ contains a nontrivial normal subgroup, say $N$, of $G.$ Since $\frat(G)=1,$ there exists a maximal subgroup $Y$ of $G$ with $N\not\leq Y.$ Since $\{M_1,\dots,M_t\}$ is a maximal irredundant family, the family $\{M_1,\dots,M_t,Y\}$ is not irredundant.
	On the other hand, since $N\leq M_1\cap\dots\cap M_t,$ we cannot have
	$M_1\cap\dots\cap M_t\leq Y.$ So, up to reordering, we may assume 
	$M_1\cap\dots\cap M_{t-1}\cap Y\leq M_t.$ Let $U=M_1\cap\dots\cap M_{t-1}.$ Since $N\leq U,$ by the Dedekind law, $N(U\cap Y)=U\cap NY=U.$
	From $U\cap Y\leq M_t,$ it follows $U=N(U\cap Y)\leq NM_t=M_t$, against the assumption that $\{M_1,\dots,M_t\}$ is an irredundant family.
\end{proof} 

\begin{question}Is the answer to question \ref{quone} negative also in the case of finite soluble groups?
	\end{question}

\section{An Example}\label{superss}

Let $H$ be a cyclic group of order $m:=p^aq^b,$ where  $p$ and $q$ two different primes and $2\leq a \leq b.$ For any divisor $r$ of $m,$ denote by $H_r$ the unique subgroup of $H$ of order $r.$
Let $I:=\{p,p^2,\dots,p^{a-1},q,q^2,\dots,q^{b-1}\}.$ For any $i\in I,$ let $p_i$ be a prime such that $m/i$ divides $p_i-1$ and let $X_i\cong C_{p_i}.$ We have an action of $H$ on $X_i$ with kernel  $H_i$. Let $X:=\prod_{i\in I}X_i$ and $G:=X\rtimes H.$ The maximal subgroups of $G$ are the following:
\begin{enumerate}
	\item $A_p:=X\rtimes H_{m/p},$ $A_q:=X\rtimes H_{m/q};$
	\item $B_{p^r,x}:=(\prod_{i\neq p^r}X_i)\rtimes H^x$, with $r\in \{1,\dots,a-1\}$ and $x\in X_{p^r};$
	\item $B_{q^s,x}:=(\prod_{i\neq q^s}X_i)\rtimes H^y$, with $s\in \{1,\dots,b-1\}$ and $y\in X_{q^s}.$
\end{enumerate}
Now we study the intersections of these maximal subgroups.
Notice that if $t\geq 2$ and $x_1,\dots,x_t$ are distinct elements
of $X_{p^r},$ then 
\begin{equation}
B_{p^r,x_1}\cap \dots \cap B_{p^r,x_t}=B_{p^r,x_1}\cap  B_{p^r,x_2}=\left(\prod_{i\neq p^r}X_i\right)\rtimes H_{p^r}.
\end{equation}
Similarly, if $t\geq 2$ and $y_1,\dots,x_t$ are distinct elements
of $X_{q^s},$ then 
\begin{equation}
B_{q^s,x_1}\cap \dots \cap B_{q^s,x_t}=B_{q^s,x_1}\cap  B_{q^s,x_2}=\left(\prod_{i\neq q^s}X_i\right)\rtimes H_{q^s}.
\end{equation}
Let $\mathcal Y$ be a family of maximal subgroups of $G$. For any $i\in I,$ let $\mathcal{Y}_i$ be the set of subgroups in $\mathcal Y$ of the form $B_{i,x},$ for some $x\in X_i.$ Moreover define
\begin{itemize}
	\item $I_0(\mathcal Y):=\{i\in I\mid \mathcal{Y}_i=\emptyset\},$
	\item $I_1(\mathcal Y):=\{i\in I\mid |\mathcal{Y}_i|=1\},$
	\item  $I_2(\mathcal Y):=\{i\in I\mid |\mathcal{Y}_i|>1\}.$
\end{itemize}
For any $i\in I_1(\mathcal Y),$ there exists a unique $x_{i,\mathcal Y}\in X_i$ such that $B_{i,x_{i,\mathcal Y}}\in \mathcal Y.$ Let
$$x_{\mathcal Y}:=\prod_{i\in I_1(\mathcal Y)}x_{i,\mathcal Y}.$$
Finally set 
$$\tau_p(\mathcal Y)=\min \{a,r\mid p^r\in I_2(\mathcal Y)\},
\tau_q(\mathcal Y)=\min \{b,s\mid q^s\in I_2(\mathcal Y)\}$$ and define
$$L_{\mathcal Y}:=\begin{cases}H &\text {if $A_p, B_p \notin \mathcal Y$,}\\
H_{p^{a-1}q^b} &\text {if $A_p \in \mathcal Y$,   $B_p \notin \mathcal Y,$}\\
H_{p^{a}q^{b-1}} &\text {if $A_p \notin \mathcal Y$,   $B_p \in \mathcal Y,$}\\
H_{p^{a-1}q^{b-1}} &\text {if $A_p \in \mathcal Y$,   $B_p \in \mathcal Y.$}\\
\end{cases}$$
\begin{enumerate}
	\item If $\tau_p(\mathcal Y)=a$ and $\tau_q(\mathcal Y)=b$, then
	$$\bigcap_{Y\in \mathcal Y}Y=\left(\prod_{i\in I_0(\mathcal Y)}X_i\right)L_{\mathcal Y}^{x_\mathcal Y}. 
	$$
	In this case $|\mathcal Y|=|I_1(\mathcal Y)|+|\mathcal Y\cap \{A_p, A_q\}|\leq a+b+2.$ Moreover, if $\mathcal Y$ is a maximal irredundant family, then by Proposition \ref{frate},
	$I_0(\mathcal Y)=\emptyset$, $I_1(\mathcal Y)=I$ and $A_p, A_q\in \mathcal Y.$ This implies $|\mathcal Y|=a+b.$

	\item If $\tau_p(\mathcal Y)<a$ and $\tau_q(\mathcal Y)=b$, then, setting $t=\tau_p(\mathcal Y),$ we have
	$$\bigcap_{Y\in \mathcal Y}Y=
	\left(\prod_{i\in I_0(\mathcal Y)}X_i\right)H_{p^t}^{x_\mathcal Y}. 
	$$
	Let $y_1,y_2$ be two different elements in $X_q.$ If follows from the fact that $\cap_{Y\in \mathcal Y}Y \cap X_{q,y_1}\cap X_{q,y_2}=1$ that $\mathcal Y$ cannot be a maximal irredundant family.
	
	\item If $\tau_p(\mathcal Y)=a$ and $\tau_q(\mathcal Y)<b$, then, setting $t=\tau_q(\mathcal Y),$ 
	we have
	$$\bigcap_{Y\in \mathcal Y}Y=
	\left(\prod_{i\in I_0(\mathcal Y)}X_i\right)H_{q^t}^{x_\mathcal Y}. 
	$$
		Let $y_1,y_2$ be two different elements in $X_p.$ If follows from the fact that $\cap_{Y\in \mathcal Y}Y \cap X_{p,y_1}\cap X_{p,y_2}=1$ that $\mathcal Y$ cannot be a maximal irredundant family.
	\item If $\tau_p(\mathcal Y)<a$ and $\tau_q(\mathcal Y)<b$, then
	$$\bigcap_{Y\in \mathcal Y}\!Y\!=
	\prod_{i\in I_0(\mathcal Y)}\!\!X_i.
	$$
Moreover, if $\mathcal Y$ is a maximal irredundant family, then by Proposition \ref{frate},
$I_0(\mathcal Y)=\emptyset$, $I_1(\mathcal Y)=I\setminus\{p^t,q^u\},$ $A_p, A_q\notin \mathcal Y.$ This implies $|\mathcal Y|=a+b.$
\end{enumerate} 
From the previous discussion, it follows:
\begin{prop}
	$\mindim(G)=\maxdim(G)=a+b.$
\end{prop}
Clearly we have an unrefinable chain in $\mathcal M(G)$ of length $|I|=a+b-2$ from $H$ to $G$ and a chain of length 2 from $H_{p^{a-1}q^{b-1}}$ to $H$. Moreover we have the following two unrefinable chains from $1$ to $H_{p^{a-1}q^{b-1}}:$
$$1<H_p\leq \dots \leq H_{p^{a-1}} \leq H_{p^{a-1}q^{b-1}},\quad \quad
1<H_q\leq \dots \leq H_{q^{b-1}} \leq H_{p^{a-1}q^{b-1}}.$$
In particular we may easily conclude:
	\begin{prop}
		$\menta(G)=2a+b, \quad \manta(G)=a+2b.$
	\end{prop}

\section{Proof of Theorem \ref{facile} and further considerations}
\begin{lemma}\label{sub}Assume that $\{M_1,\dots,M_n\}$ is a family of maximal subgroups of $G.$ There exists $J\subseteq I:=\{1,\dots,n\},$ such that 
	\begin{enumerate}
		\item $\{M_j\mid j\in J\}$ is an irredundant family;
		\item $\cap_{i\in I}M_j=\cap_{j\in J}M_j.$
	\end{enumerate}
\end{lemma}
\begin{proof}
	By induction on $n.$ If $\{M_i\mid i\in I\}$ is a redundant family, then there exists $k\in I$ such that $\cap_{i\in I}M_i=\cap_{j\in I\setminus \{k\}}M_j.$ We substitute the original family with the subfamily $\{M_j\mid j\in I\setminus \{k\}\}$ and we conclude by induction.
\end{proof}

\begin{prop}\label{tra1}
	$\mindim(G)\leq \menta(G).$
\end{prop}
\begin{proof}
	We may assume $\frat(G)=1.$ Let $t=\menta(G)$ and assume that $$\mathcal C: K_t<K_{t-1}<\dots<K_1<K_0=G$$ is a non refinable chain in $\mathcal M(G).$ There exists $t$ maximal subgroups $M_1,\dots, M_t$ of $G,$ such that $K_i=\cap_{j\leq i}M_j$ for $1 \leq i \leq t.$ Since $\frat(G)=1,$ if $K_t\neq 1,$ then there exists a maximal subgroup $M$ not containing $K_t,$ and consequently $K_t\cap M < K_t <\dots < K_0$ is a refinement of $\mathcal C.$ Hence $K_t=1.$ It follows from Lemma \ref{sub} that there exists $J\subseteq \{1,\dots,t\}$ such that
	$\{M_j\mid j\in J\}$ is an irredundant family of maximal subgroups of $G$ with $\cap_{j\in J}M_j=1.$ The second condition implies that $\{M_j\mid j\in J\}$ is a maximal irredundant family of maximal subgroups of $G,$ hence $\mindim(G)\leq t.$
\end{proof}

\begin{prop}\label{tra2}
	$\maxdim(G)\leq \manta(G).$
\end{prop}
\begin{proof}
	Let $t=\maxdim(G)$ and suppose that $\{M_1,\dots,M_t\}$ is an irredundant family of maximal subgroups of $G$. For $1\leq j\leq t,$ set $K_j=\cap_{i\leq j}M_i:$ $$K_t<K_{t-1}<\dots <K_1<K_0$$ is a chain in $\mathcal M$ of length $t,$ and this implies $\manta(G)\geq t.$
\end{proof}

It is more difficult to compare $\maxdim(G)$ and $\menta(G).$ In the example discussed in the  Section \ref{superss},
$\menta(G)-\maxdim(G)=a.$ On the other hand, if $G=\perm(p)$ with $p$ a prime, we may consider the chain $1<\GL(1,p)<\AGL(1,p)<\perm(p).$ Since $\GL(1,p)$ is maximal in $\AGL(1,p)$ and $\AGL(1,p)$ is maximal in $\perm(p),$ we may refine this chain to a chain of maximal intersections of length at most $2+\log_2(p-1).$ Since $\maxdim(\perm(p))=m(G)=p-1,$ we have examples of finite groups $G$ for which the difference $\maxdim(G)-\menta(G)$ is arbitrarily large.
 We may also construct  finite soluble groups $G$ with  $\maxdim(G)>\menta(G).$
Indeed assume that $p, q$ an $r$ are three primes
and that $p$ divides $r - 1.$
Let $\FF$ be the field with $r$ elements and let $C=\langle c \rangle $ be the subgroup of order $p$ of the multiplicative group of $\FF.$
Let $V=\FF^q$
be a $p$-dimensional vector space over $\FF$ and let $\sigma=(1,2,\dots,q)\in \perm(q).$  The wreath group $H=C\wr \langle \sigma \rangle$ has an irreducible action on $V$ defined as follows:
if $v=(f_1,\dots,f_p)\in V$ and $h=(c_1,\dots,c_p)\sigma \in H$, then $v^h=(f_{1\sigma^{-1}}c_{1\sigma^{-1}},\dots,f_{q\sigma^{-1}}c_{q\sigma^{-1}}).$
We consider the semidirect product
$G_{q,p,r}=V\rtimes H.$
Let $$e_1=(1,0,\dots,0), e_2=(0,1,\dots,0),\dots, e_q=(0,0,\dots,1) \in V,$$ $$h_1=(c,1,\dots,1), h_2=(1,c,\dots,1),\dots, h_q=(1,1,\dots,c) \in C^q\leq H.$$
For any $1\leq i,j \leq q,$ we have
$$h_i^{e_j}=h_i \text { if } i\neq j, \quad h_i^{e_i}=((1/c-1)e_i)h_i.$$
But then, for each $i\in \{1,\dots,q\},$  we have
$$h_i \in \bigcap_{j\neq i}H^{e_j}, \quad h_i \notin H^{e_i},$$
hence $H^{e_1},\dots,H^{e_q}$ is an irredundant family of maximal subgroups of $G_{q,p,r}$ and therefore
 $\maxdim(G_{p,q,r})\geq q.$ Now assume that $p$ has order $q-1$ mod $q:$ in that case $C^q$ is the direct sum of the irreducible $\langle \sigma \rangle$-module, $C_1$ and $C_2$, of dimension, respectively, $1$ and $q-1.$ If we consider $$
Y_0=1 < Y_1=\langle \sigma \rangle< Y_2= C_1 \langle \sigma\rangle < Y_3=H < Y_4=G_{p,r,q},$$
we have that $Y_i$ is a maximal subgroup of $Y_{i+1},$ so $\menta(G)\leq 4.$ Hence  the difference $\maxdim(G)-\menta(G)$ can be arbitrarily large even in the soluble case.

\

The fact that $m(G)\leq \maxdim(G)$ motivates the following question.

\begin{question}\label{quat} Does there exist a finite soluble group $G$ with the property that $m(G)>\menta(G)?$
\end{question}

We are going to prove that the previous question has an affirmative answer if the Fitting length of $G$ is at most 2. But the question remains open in the general case.

\begin{lemma}\label{parag}Let $G$ be a finite nilpotent group and let $\mathcal F$ be a family of subgroups of $G$ which contains $G$ and all the maximal subgroups of $G$ and is closed under taking intersections. Let $\mathcal C= X_t<X_{t-1}<\dots X_1<X_0$ be a chain in $\mathcal F$ . If $\mathcal C$ cannot be refined in $\mathcal F$, then $t\geq u,$ where $u$ is the composition length of $G/\frat(G).$
\end{lemma}
\begin{proof}
Clearly, since $\mathcal C$ is not refinable, we must have $X_0=G.$ Let $F=\frat(G).$ If $X_t\not\leq F,$ then there exists a maximal subgroup $M$ of $G$ not containing $X_t$. But then $X_t\cap M < X_t$. However $X_t\cap M\in \mathcal F,$ and this would contradict the assumption that $\mathcal C$ cannot be refined in $\mathcal F.$ For any $H\leq G,$ let $\overline H:=HF/F.$ We have a chain $$\overline {\mathcal C}: \ F=\overline X_t\leq \dots \leq \overline X_0.$$ Assume that $\overline X_i\neq \overline X_{i+1}.$ This implies that there exists a maximal subgroup $M$ of $G$ containing $X_{i+1}$ but not $X_i$. We have $X_{i+1}\leq X_i\cap M<X_i.$ Since $X_i\cap M\in \mathcal F$ and  $\mathcal C$ cannot be refined in $\mathcal F$, we deduce $X_{i+1}=X_i\cap M.$ But then $X_i/X_{i+1}=X_i/(X_i\cap M)\cong X_iM/M=G/M,$ hence $X_{i+1}$ is a maximal subgroup of $X_i$ (and therefore $\overline X_{i+1}$ is a maximal subgroup of $\overline X_i).$ This implies that the length of $\overline {\mathcal C}$ is precisely $u,$ hence $t\geq u.$
\end{proof}

\begin{thm}\label{fit2}
Let $G$ be a finite group. If $G/\fit(G)$ is nilpotent, then $$\menta(G)\geq m(G).$$
\end{thm}

\begin{proof}
We may assume $\frat(G)=1.$ In this case $\fit(G)=\soc(G)$ has a complement, say $X$, in $G$ and $Z(G)$ has a complement, say $W$, in $\fit G$, which is normal in $G$. Let $H=Z(G)X.$ We have$$G=W \rtimes H=\left(V_1^{\delta_1}\times \dots \times V_t^{\delta_t}\right)\rtimes H,$$ where $V_1,\dots,V_t$ are faithful irreducible $H$-modules, pairwise not $H$-isomorphic.
By \cite[Theorem 1]{mg1}, $m(G)$ is the number of non-Frattini factors in a chief series of $G.$  
 Hence $m(G)=\delta_1+\dots+\delta_t+u,$ where  $u$ be the composition length of $H/\frat(H).$
 Let $\mathcal M$ be the family of the maximal subgroups of $H$ and let $$\mathcal D=
\{C_H(v)\mid v\in V_i, 1\leq i \leq t\}.$$ Consider the family $\mathcal F$ consisting of $H$ and all  the possible intersections of elements of $\mathcal D$ and $\mathcal M.$

Now let $1=Y_\rho<\dots < Y_0=G$ be a chain of maximal intersections in $G$ that cannot be refined. By an iterated application of \cite[Theorem 15]{ik}, there exists $w\in W$ such that 
for $1\leq i\leq \rho$, we have $Y_i=U_iZ_i^w,$ with $U_i\leq_H W$ and $Z_i\in \mathcal F.$ Moreover either $Z_i=Z_{i+1}$ and $U_i/U_{i+1}\cong_H V_j$ for some $1\leq j\leq t$ or $U_i=U_{i+1}$ and $Z_{i+1}=Z_i\cap X$ with $X\in \mathcal D \cup \mathcal M.$ In the second case, the fact that there is no maximal intersection in $G$ strictly between $Y_{i+1}$ and $Y_i$ implies that $Z_{i+1}$ is $\mathcal F$-maximal in $Z_i.$ Let $J:=\{j\mid U_{j+1}<U_{j}\}$. Assume $|J|=a$ and order the elements of $J$ so that $j_1<j_2<\dots <j_a$: we have that $0<U_{j_a}<U_{j_{a-1}}\dots <U_{j_2}<U_{j_1}=W$ is an $H$-composition series of $W$ and this implies $a=\delta_1+\dots+\delta_t.$ Now let $J^*:=\{j\mid U_{j+1}=U_{j}\}$. Assume $|J^*|=b$ and order the elements of $J$ so that $j_1<j_2<\dots <j_b$: we have that $1<U_{j_b}<U_{Z_{b-1}}\dots <Z_{j_2}<Z_{j_1}=H$ is an $\mathcal F$-chain that cannot be refined, so by Lemma \ref{parag}, $b\geq u$. We conclude $\rho=a+b\geq \delta_1+\dots+\delta_t+u=m(G).$
\end{proof}

\noindent {\bf{Remark}.} One could hope to generalize Lemma \ref{parag} as follows: let $G$ be a finite soluble group and let $\mathcal F$ be a family of subgroups of $G$ which contains $G$ and all the maximal subgroups of $G$ and is closed under taking intersections. Let $\mathcal C= X_t<X_{t-1}<\dots X_1<X_0$ be a chain in $\mathcal F$. If $\mathcal C$ cannot be refined in $\mathcal F$, then $t\geq m(G).$ This would allows to prove Theorem \ref{fit2} for arbitrary finite soluble groups. However this more general statement is false. Let $G=H_1\times H_2\times H_3,$ with $H_i\cong\perm(3)$ and let $K_i$ be the Sylow 3-subgroup of $H_i.$ Let $1\neq k_i\in K_i$ and let $X=\langle k_1,k_2,k_3\rangle.$ Let $\mathcal F$ be the family of subgroups of $G$ consisting of $G,$ the maximal intersections in $G$ and $X.$ A maximal subgroup of $G$ containing $X$, contains also $K:=K_1\times K_2\times K_3,$ so $1<X<K<K_1\times K_2\times H_3<K_1\times H_2\times H_3=G$ in a non-refinable $\mathcal F$-chain in $G$. However $m(G)=6.$

\section{Strongly and weakly minmax finite groups}

\begin{lemma}\label{quofra}
	Let $N$ be a normal subgroup of a strongly minmax finite group. If $N$ in an intersection of maximal subgroups of $G$ and $G$ is weakly minmax, then $G/N$ is weakly minmax.
\end{lemma}
\begin{proof}  Let $\mathcal M^+(G,N)$ (resp. $\mathcal M^-(G,N)$ be the sublattice of $\mathcal M(G)$ consisting of the subgroups in $\mathcal M(G)$ containing $N$ (resp. contained in $N$). Let $N=H_0\leq \dots \leq H_t=G$ and 
	$N=K_0\leq \dots \leq K_u=G$	be two unrefinable chains in $\mathcal M^+(G,N)$. If $T_0\leq \dots \leq T_v=N$ is an unrefinable chain in $\mathcal M^-(G,N),$ then
	$$1=T_0\leq \dots \leq T_v=H_0\leq \dots \leq H_t=G$$ and
	$$1=T_0\leq \dots \leq T_v=K_0\leq \dots \leq K_u=G$$
	are two unrefinable chains in $\mathcal M(G)$, so, since $G$ is weakly minmax, they have the same length, but then $t=u$ and therefore $\menta(G/N)=\manta(G/N).$

	Now assume that $X_1,\dots,X_a$ is a family of maximal subgroups of $G$ with minimal size with respect to the property $X_1\cap \dots \cap X_a=N$ and that $Y_1,\dots,Y_b$ is a family of maximal subgroups of $G$ containing $N$ and  with the property that the chain $$\mathcal C: Y_1\cap \dots \cap Y_b < Y_1\cap \dots \cap Y_{b-1}<\dots <Y_1\cap Y_2 < Y_1<G$$  cannot be refined in $\mathcal M^+(G,N).$ We want to prove that $a=b.$ First, notice that, since $N$ is a maximal intersection and $\mathcal C$ is not refinable, $Y_1\cap \dots \cap Y_b=N.$ There exists $Z_1,\dots,Z_c$ such that $X_1,\dots,X_a,Z_1,\dots,Z_c$ is a maximal irredundant family of maximal subgroups of $G$. 
	Consider the chain of maximal intersections: 
	$$\begin{aligned} \ & X_1\cap \dots \cap X_a\cap Z_1\cap\dots \cap Z_c < X_1\cap \dots \cap X_a\cap Z_1\cap\dots \cap Z_{c-1}< \\ &<\dots <  X_1\cap \dots \cap X_a \cap Z_1 <  X_1\cap \dots \cap X_a < \dots < X_1<G.
	\end{aligned}$$
	Since $G$ is weakly minmax, this chain cannot be refined (and in particular $a+c=\menta(G)=\manta(G)$). On the other hand
	$$\begin{aligned}\ & N\cap Z_1\cap\dots \cap Z_c < N\cap Z_1\cap\dots \cap Z_{c-1}< \\ &<\dots <  N\cap Z_1 <  N= Y_1\cap \dots \cap Y_b < \dots < Y_1<G
	\end{aligned}$$
	is a chain of maximal intersections and we must have $c+b\leq \manta(G)=a+c,$ hence $b\leq a$ (and consequently $b=a$).
\end{proof}

\begin{thm}\label{unswm}
	A finite weakly minmax group is soluble.
\end{thm}

Before to prove this theorem, we need to introduce a couple of definitions and related lemmas.

\begin{defn}
	Let $X$ be an almost simple group and $S=\soc X.$
	\begin{enumerate}
		\item We define $\sigma(X)$ as the largest positive integer $\sigma$ for which there exists a core-free maximal subgroup $Y$ of $X$ and $s_1,s_2,\dots,s_\sigma$ in $S$ such that $$Y^{s_1}\cap S > Y^{s_1}\cap Y^{s_2}\cap S > \dots > Y^{s_1}\cap Y^{s_2}\cap \dots \cap Y^{s_\sigma}\cap S.$$
		\item We define $\tau(X)$ as the minimal size of a family of core-free maximal subgroups of $X$ with trivial intersection.
	\end{enumerate}
\end{defn}

\begin{lemma}\label{frob}
	$\sigma(X)\geq 3.$
\end{lemma}
\begin{proof}
	Let $Y$ be a core-free maximal subgroup of $X$ and let $T=S\cap Y$. We have $T\neq 1$ (see for example the last paragraph of the proof of the main theorem in \cite{lps}). If suffices to prove that there exists $s\in S$ such that $1<T\cap T^s<T.$ We have $Y=N_G(T)$, so $T=N_S(T).$
	Assume by contradiction $1=T\cap T^s$ for every $s\in S\setminus T$: this means that $S$ is a Frobenius group and $T$ is a Frobenius complement, but this is in contradiction with the fact that $S$ is a non-abelian simple group.
\end{proof}

\begin{lemma}
	\label{tau}$\tau(X)\leq 4,$ with equality if and only if $X=U_4(2).2$
\end{lemma}
\begin{proof}
	See \cite[Theorem 1]{bgl2}.
\end{proof}

\begin{proof}[Proof of Theorem \ref{unswm}]
	Let $G$ a finite weakly minmax group. If $G$ is not soluble, then it admits a non-abelian chief factor $H/K.$ Let $C=C_G(H/K).$ Then $G/C$ is a monolithic group (with socle isomorphic to $H/K$) and is weakly minmax by Lemma \ref{quofra}. So in order to complete the proof, it would suffice to prove that a finite monolithic group with non-abelian socle cannot be weakly minmax.

	Let $G$ be a monolithic primitive group, and assume $N=\soc(G)\cong S^n,$ with $S$ a non-abelian simple group.
	Let  $\psi$ be the map from  $N_G(S_1) $ to  $\aut(S)$ induced by
	the conjugacy  action on $S_1$.
	Set $X=\psi(N_G(S_1))$ and note that $X$ is an almost simple group with socle
	$S=\inn(S)=\psi(S_1)$.
	Let $T:=\{t_1,\ldots,t_n\}$
	be a right transversal of $N_G(S_1)$ in $G;$
	the map $$\phi_T: G \to
	X \wr \sym(n)$$ given by
	$$g \mapsto ( \psi(t_{1}^{} g t_{1 \pi_g}^{-1}), \dots ,  \psi(t_n^{} g t_{n
		\pi_g}^{-1})) \pi_g$$
	where $\pi_g \in \sym (n)$ satisfies $t_i^{}g t_{i \pi_g}^{-1} \in
	N_G(S_1)$ for all $1\leq i\leq  n$, is an injective
	homomorphism.
	So we may identify  $G$ with
	its image in $X \wr T$, where $T=\{\pi_g\mid g\in G\}$ is a transitive subgroup of $\perm(n).$  In this identification,  $N$ is contained in
	the base subgroup $X^n$ and $S_i $ is a subgroup of the $i$-th
	component of $X^n$.

	Let $F/N=\frat(G/N)$ and assume that $Y_1,\dots,Y_t$ is a family of maximal subgroups of $G$ of minimal size with respect to the property $F=Y_1\cap \dots \cap Y_u.$ Now choose a core-free maximal subgroup $Y$ of $X$ and $s_1,\dots,s_\sigma$ as in the definition of $\sigma=\sigma(X)$ and let $M=G\cap (Y\wr T)$. By \cite{classes} Proposition 1.1.44, $M$ is a maximal subgroup of $G$. For $1\leq i\leq n$ and $1\leq j\leq \sigma,$ let $\tau_{i,j}=(1,\dots,1,s_j,1\dots,1)\in S^n,$ where $s_j$ is in the $i$-th position of $\tau_{i,j}$
	and let $M_{i,j}=M^{\tau_{i,j}}.$ We order lexicologically the pairs $(i,j)$. 
	Let $\Sigma_{k,l}=\cap_{1\leq i\leq k, 1\leq j\leq l}M_{i,j}\cap F.$
	We have $\Sigma_{k,l}\cap M_{k,l}<\Sigma_{k,l}$ and this implies
	\begin{equation}\label{aaa}\manta(G)\geq t+n\cdot \sigma(X).\end{equation}
	Now let $\tau=\tau(X)$ and suppose that $R_1,\dots,R_\tau$ are core-free maximal subgroups of $X$ with trivial intersection.
	Again by \cite{classes} Proposition 1.1.44, $Z_i:=G\cap (R_i\wr T)$ is a maximal subgroup of $G$ for $1\leq i\leq \tau.$
	Let $W=Y_1\cap \dots \cap Y_t\cap Z_1\cap \dots \cap Z_\tau$
	and let $\pi: G\to T$ the epimorphism sending $g$ to $\pi_g.$
	Since $W\cap X^n=1,$ we have $W\cap \ker \pi=1$, so $W$ is isomorphic to a (proper) subgroup of $T.$ Moreover since  $W\cap N=1,$ we have $W\cong WN/N \leq F/N,$ hence $W$ is a nilpotent subgroup of $\perm(n).$ With the same arguments used by Cameron, Solomon and Turull in \cite{cst}, it can be easily proved that the maximal length $l(K)$ of a chain of subgroups in a nilpotent permutation group  of degree $n$ is at most $n-1.$ It follows $l(W)\leq n-1.$ Since $G$ has trivial Frattini subgroup, there exist at most $n-1$ maximal subgroups of $G$ whose total intersection with $W$ is trivial, hence \begin{equation}\label{bbb}\alpha(G) \leq t+\tau(G)+n-1.
	\end{equation}
	Since $G$ is weakly minmax, combining (\ref{aaa}) and (\ref{bbb}), we get
	\begin{equation}n\cdot (\sigma(X)-1) < \tau(X)\end{equation}
	By Lemmas \ref{frob} and \ref{tau}, $2n<4,$ hence $n=1.$
	
	We have so proved that $n=1,$ i.e. $G$ is an almost simple group.
	First assume that $G=S$ is a simple group.
	We have
	$\alpha(G)=\manta(G) \geq \maxdim(G)\geq m(G)\geq 3
	$. Since, by \cite[Theorem 1]{bgl}, $\alpha(G)\leq 3,$ it follows $\maxdim(G)=m(G)=\alpha(G)=3.$
	Since $m(A_n)\geq n-2$ and $\alpha(A_5)=2$, it follows that $G$ is not an alternating group. If $G$ is sporadic, then by \cite[Theorem 3.1]{bgl} the condition ${\alpha}(G)=3$ implies that $G={\rm M}_{22}$. However ${\rm M}_{22}$ has a maximal subgroup $H = \rm L_3(4)$ with $b({\rm M}_{22},H)=5$ (see \cite[Table 1]{spo}) and we deduce that ${\manta}({\rm M}_{22})\geq 5.$ If $G$ is an exceptional group of Lie type, then \cite[Theorem 1]{bgl} implies that $G = G_2(2)' \cong {\rm U}_{3}(3),$ however, by a theorem of Wagner \cite{Wag}, $G$ can be generated by $4$ involutions and no fewer, so $m(G) \geq 4$.
	

	So we may assume $S<G.$
	Let $H$ be a core-free maximal subgroup of $G$ and let $b=b(G,H)$ be the minimal size of a set of conjugates of $H$ with trivial intersection (i.e. the base size of the primitive action of $G$ on the set of the right cosets of $H$). This set of conjugates of $H$ is an irredundant family of maximal subgroups with maximal size, so $\alpha(G)\leq b(G,H)\leq \maxdim(G)\leq\manta(G).$ In particular, if $G$ is weakly minmax, then $b(G,H)$ is the same for any choice of a core-free maximal subgroup $H$ of $G$. It follows from \cite{spo}, that if $\soc G$ is a sporadic simple group, then $G$ has faithful primitive actions with different base sizes, hence $G$ is not weakly minmax. In any case, by Lemma \ref{tau}, if $G$ is weakly minmax then $\alpha(G)=\maxdim(G)=\menta(G)=\manta(G)\leq \tau(G) = 3$ if
	$G\neq U_4(2).2,$ $\alpha(G)=\maxdim(G)=\menta(G)=\manta(G)\leq \tau(G)=4$ if $G=U_4(2).2.$
	
	If $n\geq 5,$ then $b(\perm(n),\perm(n-1))=n-1\geq 4>3$, and therefore  $\perm(n)$ is not weakly minmax. 
	Finally assume that $G$ is an almost simple with a socle $S$ of Lie type and $S<G.$
	Let $B$ be Borel subgroup of $S$ and let $u$ be the number of the nodes of the associated Dynkin diagram, or the number of the orbits for a suitable groups of symmetries of this diagram if $S$ is of twisted type or $G$ involves a graph automorphism of $G).$ 
	There exists a family $Y_1,\dots,Y_u$ of maximal parabolic subgroups of $G$ such that  $S>S\cap Y_1>S\cap Y_1\cap Y_2>\cdots>S\cap Y_1\cap Y_2\cap \dots \cap Y_u=B$. Moreover, as in the proof of Lemma \ref{frob}, since $N_S(B)=B$ and $S$ is not a Frobenius group, there exists $x\in S$ with $1<B\cap B^x<B.$ Let $m=m(G/S)$. There exists an irredundant family $X_1,\dots,X_m$ of maximal subgroup of $G$ containing $S$.
	Let $X=X_1\cap \dots \cap X_m.$
	Then $1<(X\cap Y_1\cap \dots \cap Y_t)\cap (X\cap Y_1\cap \dots \cap Y_t)^x< X\cap Y_1\cap \dots \cap Y_t < \dots < X \cap Y_2\cap Y_1
	< Y_1\cap X < X < \cdots < X_2 \cap X_1 < X_1 < G$ is a chain in $\mathcal M(G)$, so 
	$\tau(G)\geq \menta(G)\geq m+u+2\geq 3+u,$  a contradiction.
\end{proof}

\begin{prop}\label{pony}
	Let $G$ be a primitive monolithic soluble group. If $G$ is weakly minmax, then the derived length of $G$ is at most 3.
\end{prop}

\begin{proof}
	Assume $G=V\rtimes H,$ where $V$ is an irreducible $H$-module and $H$ is a finite soluble group. By \cite[Theorem 2.1]{ser}, the base size $b(H)$ of $H$ on $V$ is at most 3, i.e. there exist $v_1,v_2,v_3\in V$ such that $C_H(v_1)\cap C_H(v_2)\cap C_H(v_3)=1$. This implies $1=H\cap H^{v_1}\cap H^{v_2}\cap H^{v_3},$  hence $\alpha(G)\leq 4.$ 
	
	Let $1=X_0<X_1<\dots < X_t=G$ be a chain of normal subgroups in $\mathcal M(G)$ with the property that, for every $0\leq i\leq t-1,$ there is no normal subgroup $Y\in \mathcal M(G)$ with $X_i < Y < X_{i+1}.$  For every $0\leq i\leq t-1$, let $Y_i/X_i$ be a minimal normal subgroup of $G/X_i$ contained in $X_{i+1}/X_i.$ If $M$ is a maximal subgroup of $G$ containing $Y_i,$ then $\core_G(M)\cap X_{i+1}\in \mathcal M(G)$ is a normal subgroup containing $Y_i,$ hence $\core_G(M)\cap X_{i+1}=X_{i+1}$ and consequently $X_{i+1}\leq M.$ 
	This implies $X_{i+1}/Y_i=\frat(G/Y_i).$ In particular $Y_0=V$ and $X_1=V\frat(H).$ If we refine the normal series $X_0<Y_0<X_1<Y_1<\dots <  X_t=G$ to a chief series of $G,$ the non-Frattini factors are precisely the $t$ factors $Y_i/X_i$ for $0\leq i\leq t-1.$ In particular, by \cite[Theorem 2]{mg1}, $t=m(G).$ Since $G$ is weakly minmax,  $t=m(G)\leq \maxdim(G)\leq \manta(G)=\alpha(G)\leq 4.$ 
	
	Assume $t=4.$ In this case $\alpha(G)=m(G)=4$ and the chain $1<X_1<X_2<X_3<X_4=G$ cannot be refined inserting other maximal intersections. In particular $X_1\cap H=1$ and therefore $Y_1=X_1=V$ and $\frat(H)=1.$ If $i\geq 1,$ then $X_i=VZ_i$, with $Z_i\in \mathcal M(H)$ and $1=Z_1=H\cap X_1 <Z_2=H\cap X_2<Z_3=H\cap X_3<Z_4=H<G$ is a non refinable chain in $\mathcal M(G).$
	Since $Z_2$ is normal in $H$ and $V$ is a faithful irreducible $H$-module, we must have $C_V(Z_2)=1.$ Let $0\neq v\in V.$ Then $C_{Z_2}(v)=H^v\cap Z_2 <Z_2$ and therefore, since $H^v$ is a maximal subgroup of $G$, we must have  $C_{Z_2}(v)=1$.
In particular if $N$ is a minimal normal subgroup of $H$ contained in $Z_2,$ then $N$ is an elementary  abelian group acting fixed-point-freely on $V$,  so it is cyclic of prime order (see \cite[10.5.5]{rob}). Moreover $H$ is not a supersoluble, otherwise  there would $x$ and $y$ in $V$ such that $C_H(x)\cap C_H(y)=1$ (see \cite[Theorem A]{wolf}) and consequently
$\alpha(G)\leq 3.$ In particular $G$ contains a minimal normal subgroup $M$ which is not $H$-isomorphic to $N.$ Since $m(H)=m(G)-1=3,$ we must have $H\cong (N\times M)\rtimes K$ with $K$ a cyclic group of order $p^t$ for a suitable prime $p.$
In particular $V < VMK^p < VNMK^p < VH=G$ is a chain of normal subgroups in $\mathcal M(G)$. Arguing as before, we deduce that also $M$ acts fixed-point-freely of $V$, but this would imply that $M$ is cyclic of prime order and $H$ is supersoluble. This excludes $t=4.$

Assume $t=3.$ Then $m(H)=3$, so in particular there exist two primes $p$ and $q$ and a normal subgroup $N$ of $H$ such that $F=\frat(H)\leq N,$ $N/F$ is a non-Frattini chief factor of $H$ of $q$-power order and $H/N$ is cyclic of $p$-power order. Let $K$ be the unique maximal subgroup of $H$ containing $N.$ Then $1\leq F < VF < VK < G$ is a chain in $\mathcal M(H),$ so $\manta(G)\geq 3,$ with equality only if $F=1$.
In particular if $\alpha(G)=3$, then $F=1$ and $G$ has derived length at most 3. We remain with the case $F\neq 1$ and $\alpha(G)=4.$ This implies in particular that the minimal size $b(H)$ of $H$ on $V$ is equal to 3. By \cite[Theorem 1.1]{hp}, $(|H|,|V|)\neq 1.$ If follows that $V$ is a $q$-group and $F$ a $p$-group. Let $P$ be a Sylow $q$-subgroup of $H$. Since $P$ is contained in $N$, by the Frattini Argument,
$G=NN_G(P)=FPN_G(P)=FN_G(P)=N_G(P)$ so $P$ is normal in $G$ and $\frat(G/P)\geq PF/P = N/P$ (\cite[5.2.13 (iii)]{rob}). But then $G/P$ is a cyclic $p$-group. This would imply that $H$ is supersoluble, and consequently $b(H)=2$ by   \cite[Theorem A]{wolf}, a contradiction.

Finally, if $t\leq 2,$ then $m(H)\leq 1$ and consequently $H$ is cyclic and $G$ is metabelian.
\end{proof}

\begin{cor}
If $G$ is weakly minmax, then the derived length of $G/\frat(G)$ is at most 3.
\end{cor}
\begin{proof}
It follows immediately from the fact that $G$ can be embedded in
$\prod_{N\in \Omega}G/N$ being $\Omega$ be the family of the normal cores of the maximal subgroups of $G$ and that $G/N$ has derived length at most 3 for any $N\in \Omega$ by the previous proposition.
\end{proof}


\begin{thebibliography}{99}
	
	\bibitem{classes}	
A. Ballester-Bolinches and L.~M. Ezquerro, {Classes of finite
	groups}, Mathematics and Its Applications (Springer), vol. 584, Springer,
Dordrecht, 2006.



\bibitem{spo} T. Burness, E. O'Brien and R. Wilson, 
Base sizes for sporadic simple groups,
Israel J. Math. 177 (2010), 307–333.



\bibitem{bgl} T. Burness, M. Garonzi and A. Lucchini, On the minimal dimension of a finite simple group. With an appendex by T. C. Burnes and R. M. Guralnick. J. Combin. Theory Ser. A 171 (2020), 105175.

\bibitem{bgl2} T. Burness, M. Garonzi and A. Lucchini, Finite groups, minimal bases and the intersection number, arXiv:2009.10137.
	
\bibitem{cst} P. Cameron, R. Solomon and A. Turull,
Chains of subgroups in symmetric groups,
J. Algebra 127 (1989), no. 2, 340--352.	
	
	
\bibitem{delu}E. Detomi and A. Lucchini,  Maximal subgroups of finite soluble groups in general position. Ann. Mat. Pura Appl. (4) 195 (2016), no. 4, 1177--1183.
	
	
	\bibitem{ik}	I. De Las Heras and A. Lucchini, Intersections of maximal subgroups in prosoluble groups. Comm. Algebra 47 (2019), no. 8, 3432--3441.
	
	\bibitem{GL} M. Garonzi and A. Lucchini, {Maximal irredundant families of minimal size in the alternating group}, Arch. Math. (Basel) {113} (2019), 119--126.
	
	\bibitem{hp}  Z. Halasi and K. Podoski, Every coprime linear group admits a base of size two, Trans. Amer. Math. Soc. 368 (2016), no. 8, 5857--5887.
	
	\bibitem{mg1} A. Lucchini, 
	The largest size of a minimal generating set of a finite group,
	Arch. Math. (Basel) 101 (2013), no. 1, 1--8.
	
	\bibitem{join} A. Lucchini, Finite groups with the same join graph as a finite nilpotent group, Glasgow Mathematical Journal, First View , pp. 1--11
	DOI: https://doi.org/10.1017/S0017089520000415.
	
	
\bibitem{iwa} K. Iwasawa, {\"{U}ber die endlichen Gruppen und die Verb\"{a}nde ihrer Untergruppen}, J. Fac. Sci. Imp. Univ. Tokyo. Sect. I. {4} (1941), 171--199.
	
\bibitem{lps}	M. Liebeck, C. Praeger and J. Saxl,
	On the O'Nan-Scott theorem for finite primitive permutation groups,
	J. Austral. Math. Soc. Ser. A 44 (1988), no. 3, 389--396.


\bibitem{rob} D. Robinson, A course in the theory of groups. Second edition. Graduate Texts in Mathematics, 80. Springer-Verlag, New York, 1996.




\bibitem{ser} A. Seress, 
The minimal base size of primitive soluble permutation groups, 
J. London Math. Soc. (2) 53 (1996), no. 2, 243--255.

\bibitem{Wag} A. Wagner, {The minimal number of involutions generating some finite three-dimensional groups}, Boll. Un. Math. Ital. {15} (1978), 431--439. 



	
		\bibitem{wolf} T. R. Wolf, 
	Large orbits of supersolvable linear groups,
	J. Algebra 215 (1999), no. 1, 235--247. 

\end{thebibliography}
\end{document}